\documentclass{amsart}

\usepackage{amsmath}
\usepackage{amssymb}
\usepackage{enumitem}
\usepackage{hyperref}

\newtheorem{theorem}{Theorem}[section]
\newtheorem{lemma}[theorem]{Lemma}
\newtheorem{proposition}[theorem]{Proposition}

\theoremstyle{definition}

\theoremstyle{remark}
\newtheorem{remark}[theorem]{Remark}

\numberwithin{equation}{section}

\newcommand{\R}{\mathbb{R}}
\newcommand{\C}{\mathbb{C}}
\newcommand{\CC}{\mathcal{C}}
\newcommand{\SO}{\mathrm{SO}}
\newcommand{\wtSO}{\widetilde{\mathrm{SO}}}
\newcommand{\OO}{\mathcal{O}}
\newcommand{\HH}{\mathcal{H}}
\newcommand{\GG}{\mathcal{G}}
\newcommand{\VV}{\mathcal{V}}
\newcommand{\LL}{\mathcal{L}}
\newcommand{\g}{\mathfrak{g}}
\newcommand{\h}{\mathfrak{h}}
\newcommand{\z}{\mathfrak{z}}
\newcommand{\so}{\mathfrak{so}}
\newcommand{\spi}{\mathfrak{sp}}
\newcommand{\gex}{\mathfrak{g}_{2(2)}}
\newcommand{\Gex}{G_{2(2)}}

\newcommand{\Iso}{\operatorname{Iso}}
\newcommand{\Kill}{\operatorname{Kill}}

\newcommand{\rad}{\operatorname{rad}}
\newcommand{\End}{\operatorname{End}}

\begin{document}

\title{Pseudo-Riemannian $\Gex$-manifolds with dimension at most $21$}

\author{Raul Quiroga-Barranco}
\address{Centro de Investigaci\'on en Matem\'aticas, Apartado Postal
  402, Guanajuato, Guanajuato, 36250, Mexico}
\email{quiroga@cimat.mx}

\thanks{This research was supported by a Conacyt grant and by SNI}


\maketitle

\begin{abstract}
    Let $\Gex$ be the non-compact connected simple Lie group of type $G_2$ over $\R$, and let $M$ be a connected analytic complete pseudo-Riemannian manifold that admits an isometric $\Gex$-action with a dense orbit. For the case $\dim(M) \leq 21$, we provide a full description of the manifold $M$, its geometry and its $\Gex$-action. The latter are always given in terms of a Lie group geometry related to $\Gex$, and in one case $M$ is essentially the quotient of $\SO_0(3,4)$ by a lattice.
\end{abstract}

\section{Introduction}
\label{sec:intro}
A fundamental problem in both geometry and dynamics is to understand the actions of a connected simple Lie group $G$ on manifolds. This is particularly interesting when one of such $G$-actions preserves a geometric structure on a manifold $M$. A basic example to consider is the left $G$-action on the manifold $H/\Gamma$ where $H$ is a semisimple Lie group, $\Gamma$ is a lattice of $H$ and the action is given by a non-trivial homomorphism $G \rightarrow H$. There are two well known properties for such an example. In the first place, the $G$-action on $H/\Gamma$ is ergodic and so has a dense orbit. And secondly, the Killing form of the Lie algebra of $H$ induces a bi-invariant pseudo-Riemannian metric on $H$ that descends to a metric on $H/\Gamma$ preserved by the $G$-action.

It has been conjectured that every finite volume preserving ergodic $G$-action on a manifold is essentially one of the examples $H/\Gamma$ just described (see \cite{ZimmerProgram}). In this direction, Zimmer's program proposes to study ergodic $G$-actions to understand their rigid properties. Many efforts on this line of research have shown very useful to consider actions that preserve a geometric structure. In particular, this has lead to the development of a set of tools widely known as Gromov-Zimmer's machinery (see for example \cite{Gromov,ZimmerAut,QuirogaCh}).

In this work we consider the case $G = \Gex$, the connected non-compact exceptional Lie group of type $G_2$ over $\R$, and an isometric $\Gex$-action on a finite volume pseudo-Riemannian manifold $M$. Following Zimmer's program, the final goal is to prove that $M$ is closely related to the group $\Gex$ itself, from the viewpoint of all the structures involved. We achieve this objective for the case $\dim(M) \leq 21$. We note that $21$ is the dimension of the Lie group $\SO(3,4)$ and that there is a homomorphism $\Gex \rightarrow \SO(3,4)$ that realizes the irreducible non-trivial representation of $\Gex$ with lowest dimension (see Section~\ref{sec:intro} for details). The following result proves that $M$ is always related to $\Gex$ and in one case that it is essentially given by a non-trivial homomorphism $\Gex \rightarrow \SO(3,4)$. In a sense, we thus provide a geometric/dynamic characterization of the homomorphism $\Gex \rightarrow \SO(3,4)$ that defines the irreducible $7$-dimensional representation of $\Gex$.

\begin{theorem}[Manifold and action type]
    \label{thm:diff_type}
    Let $(M,h)$ be a connected analytic complete pseudo-Riemannian manifold with finite volume that admits an isometric $\Gex$-action with a dense orbit. If $\dim M \leq 21$, then there exists a finite covering map $\pi : \widehat{M} \rightarrow M$ so that $\widehat{M}$ satisfies one of the following properties.
    \begin{enumerate}
        \item There exist a connected pseudo-Riemannian manifold $N$ and a discrete subgroup $\Gamma \subset \Gex \times \Iso(N)$ such that
            \[
                \widehat{M} = \left(\Gex \times N\right)/\Gamma.
            \]
            Furthermore, the $\Gex$-action on $M$ lifted to $\widehat{M}$ is precisely the left $\Gex$-action induced from the action on the first factor of $\Gex \times N$.
        \item There exist a lattice $\Gamma \subset \wtSO_0(3,4)$ so that
            \[
                \widehat{M} = \wtSO_0(3,4)/\Gamma.
            \]
            Furthermore, the $\Gex$-action on $M$ lifted to $\widehat{M}$ is precisely the left $\Gex$-action induced from a non-trivial homomorphism $\Gex \rightarrow \wtSO_0(3,4)$ and the left translation action on $\wtSO_0(3,4)$.
    \end{enumerate}
\end{theorem}

The next result proves that the pseudo-Riemannian metric on $M$ can also be related to natural metrics.

\begin{theorem}[Metric type]
    \label{thm:metric_type}
    With the hypotheses and notation of Theorem~\ref{thm:diff_type}, one of the following holds according to the cases of such theorem.
    \begin{enumerate}
        \item For $\widehat{M} = (\Gex \times N)/\Gamma$, the covering map $\pi : (\widehat{M},\widehat{h}) \rightarrow (M,h)$ is locally isometric for the metric $\widehat{h}$ on $\widehat{M}$ induced from the product metric on $\Gex\times N$ where $\Gex$ carries a bi-invariant metric.
        \item For $\widehat{M} = \wtSO_0(3,4)/\Gamma$, there is new $\Gex$-invariant metric $\overline{h}$ on $M$ so that the covering map $\pi : (\widehat{M},\widehat{h}) \rightarrow (M,\overline{h})$ is locally isometric for the metric $\widehat{h}$ on $\widehat{M}$ induced from the bi-invariant metric on $\wtSO_0(3,4)$ given by the Killing form of $\so(3,4)$.
    \end{enumerate}
\end{theorem}

It is well known that the group $\wtSO_0(3,4)$ is weakly irreducible for the bi-invariant metric defined by the Killing form of its Lie algebra. We recall that a pseudo-Riemannian manifold is weakly irreducible if its not locally isomorphic to a product of pseudo-Riemannian manifolds. This property can be used to distinguish between the two cases of the theorems above.

As for the organization of the work, in Section~\ref{sec:G22} we present some basic facts on $\Gex$, its Lie algebra $\gex$ and their representations. Section~\ref{sec:centralizer} applies the Gromov-Zimmer's machinery to describe the centralizer $\HH$ of the $\Gex$-action on the universal covering space $\widetilde{M}$ in the Lie algebra of Killing fields. Finally, Section~\ref{sec:main_proofs} provide the proofs of the results stated in this Introduction.

\section{Preliminaries on $\Gex$}
\label{sec:G22}
We introduce the exceptional Lie group $\Gex$ and recall some properties that we will use in this work. This includes some properties of $\gex$, the Lie algebra of $\Gex$. We refer to \cite{Jacobson,SpringerVeldkamp} for further details.

We define $\Gex$ as the connected group of automorphisms of the split Cayley algebra $\CC$ over $\R$. We recall that $\CC$ is a composition algebra whose norm is a split quadratic form. In other words, the norm of $\CC$ is a quadratic form whose associated bilinear form has signature $(4,4)$. The group $\Gex$ preserves the unit $e$ of $\CC$ and so it preserves the orthogonal complement $e^\perp$ which is precisely the space of pure imaginary elements of $\CC$: the set of $a \in \CC$ such that $\overline{a} = -a$. The bilinear form of $\CC$ restricted to $e^\perp$ has signature $(3,4)$ and so we will denote $e^\perp = \R^{3,4}$. This yields a faithful representation $\Gex \rightarrow \SO(3,4)$, that we will call the linear realization of $\Gex$. Correspondingly, there is a Lie algebra representation $\gex \rightarrow \so(3,4)$, that exhibits $\gex$ as the Lie algebra of derivations of $\CC$ restricted to $e^\perp$. We will call this representation the linear realization of $\gex$.

\begin{proposition}
    \label{prop:g22_and_R34}
    The Lie algebra $\gex$ is the split exceptional Lie algebra of type $G_2$ over $\R$. The linear realization $\gex \rightarrow \so(3,4)$ turns the space $\R^{3,4}$ into an irreducible $\gex$-module. Furthermore, $\R^{3,4}$ is the $\gex$-module corresponding to the first fundamental weight of $\gex$ and every other irreducible $\gex$-module has dimension at least $14$.
\end{proposition}
\begin{proof}
    The first claim is well known (see \cite{Jacobson,SpringerVeldkamp}).

    Next, we recall that the irreducible representations of the split form $\g$ of a simple complex Lie algebra $\g^\C$ are all real forms of the corresponding irreducible representations of $\g^\C$ (see \cite{Oni}). On the other hand, the irreducible representation of $\g_2^\C$ corresponding to the first fundamental weight has dimension $7$, and all other irreducible non-trivial representations have dimension at least $14$ (see \cite{Humphreys}). Hence, the irreducible representation of $\gex$ associated to the first fundamental weight has (real) dimension $7$. Since the linear realization homomorphism $\gex \rightarrow \so(3,4)$ is non-trivial, it defines such irreducible representation.

    Finally, we recall that Weyl's dimension formula implies that the irreducible representations of $\g_2^\C$ corresponding to the first and second fundamental weights have dimensions $7$ and $14$, and that every other irreducible representations has dimension strictly larger (see \cite{Humphreys}). Hence, the last claim follows from the above remarks on split forms.
\end{proof}

In the rest of this work, $\R^{3,4}$ will be considered as a $\gex$-module with the structure given by Proposition~\ref{prop:g22_and_R34}.

The following result establishes the uniqueness of the $\gex$-invariant scalar product on $\R^{3,4}$.

\begin{proposition}
    \label{prop:R34_scalar_product}
    The $\gex$-module $\R^{3,4}$ carries a unique, up to a constant multiple, scalar product invariant under $\gex$. In particular, any such scalar product has signature either $(3,4)$ or $(4,3)$.
\end{proposition}
\begin{proof}
    The existence of the scalar product is clear from the construction of the $\gex$-module $\R^{3,4}$ in terms the composition algebra $\CC$.

    Recall that there is a natural isomorphism of vector spaces between the space of $\gex$-invariant bilinear forms on $\R^{3,4}$ and the real algebra $\End_{\gex}(\R^{3,4})$ of homomorphisms of $\gex$-modules of $\R^{3,4}$. Hence, it is enough to prove that $\End_{\gex}(\R^{3,4})$ is $1$-dimensional.

    By Schur's Lemma and the irreducibility of $\R^{3,4}$, the algebra $\End_{\gex}(\R^{3,4})$ is a division algebra over $\R$, and so it is isomorphic to either $\R$, $\C$ or the quaternion numbers. If $\End_{\gex}(\R^{3,4})$ is not $1$-dimensional, then there is a $\gex$-invariant complex structure on $\R^{3,4}$, which is absurd since this space is odd-dimensional.
\end{proof}

We recall the following elementary property.

\begin{lemma}
    \label{lem:so(E)}
    Let $E$ be a finite dimensional vector space with scalar product $\left<\cdot, \cdot\right>$. Then, the assignment
    \[
        u\wedge v \mapsto \left<\cdot,u\right> v - \left<\cdot,v\right> u,
    \]
    defines an isomorphism $\varphi : \wedge^2 E \rightarrow \so(E)$ of $\so(E)$-modules. In particular, $\varphi$ yields an isomorphism of $\g$-modules for every Lie subalgebra $\g$ of $\so(E)$.
\end{lemma}

As a consequence $\wedge^2 \R^{3,4} \simeq \so(3,4)$ as $\gex$-modules for the structures defined by the linear realization of $\gex$. The next proposition describes some useful properties of these $\gex$-modules.

\begin{remark}
    \label{rmk:brackets_and_modules}
    If $\h$ is a Lie algebra, then the Jacobi identity implies that the linear map
    \begin{align*}
        \wedge^2 \h &\rightarrow \h \\
            X \wedge Y &\mapsto [X,Y],
    \end{align*}
    is a homomorphism of $\h$-modules. In particular, if $\h_1$ is a Lie subalgebra of $\h$ and $V_1, V_2$ are $\h_1$-submodules of $\h$ (for the $\h_1$-module structure defined by the Lie brackets), then $[V_1, V_2]$ is an $\h_1$-module (again, by the Jacobi identity) whose irreducible $\h_1$-submodules must be among those that appear in $V_1 \otimes V_2$. Similarly, there is a corresponding remark for $[V_1, V_1]$ and $\wedge^2 V_1$. We will use these facts in the rest of this work.
\end{remark}

\begin{proposition}
    \label{prop:wedge2R34_so(R34)_over_g22}
    The following isomorphism of $\gex$-modules holds
    \[
        \wedge^2 \R^{3,4} \simeq \so(3,4) \simeq \R^{3,4} \oplus \gex,
    \]
    where $\gex$ is the $\gex$-module given by the adjoint representation. If we let $V$ denote the $\gex$-submodule of $\so(3,4)$ isomorphic to $\R^{3,4}$, then
    \[
        [V, V] = \so(3,4),
    \]
    with respect to the Lie brackets of $\so(3,4)$.
\end{proposition}
\begin{proof}
    Since $\gex$ is $\gex$-submodule of $\so(3,4)$ (for the structure mentioned above) and since $\gex$ is simple, there is a $\gex$-submodule $V$ of $\so(3,4)$ such that
    \[
        \so(3,4) = \gex \oplus V.
    \]
    In particular, $V$ has dimension $7$. By Proposition~\ref{prop:g22_and_R34}, either $V$ is a direct sum of trivial $1$-dimensional modules or $V \simeq \R^{3,4}$ as $\gex$-modules. If the former occurs, then Remark~\ref{rmk:brackets_and_modules} implies that $[V,V]$ is a sum of $1$-dimensional $\gex$-modules as well and so $[V,V] \subset V$. This implies that $V$ is a proper ideal of $\so(3,4)$, which is absurd. This proves the first claim.

    On the other hand, by Remark~\ref{rmk:brackets_and_modules} the space $[V,V]$ is either $0$ or a sum of the irreducible $\gex$-modules that appear in $\so(3,4)$. We have already ruled out that $[V,V] \subset V$. If $[V,V] \subset \gex$, then $(\so(3,4), \gex)$ is a symmetric pair. But an inspection of Table~II from \cite{Berger} shows that no such symmetric pair exists. Therefore, the only possibility left is to have $[V,V] = \so(3,4)$.
\end{proof}

\begin{remark}
    \label{rmk:wedge2R34_so(R34)_over_g22}
    Note that $\gex$, as a module over itself, is the irreducible representation of $\gex$ corresponding to the second fundamental weight. Hence, Proposition~\ref{prop:wedge2R34_so(R34)_over_g22} says that the $\gex$-module $\wedge^2 \R^{3,4} \simeq \so(3,4)$ is the sum of the irreducible representations corresponding to the fundamental weights.
\end{remark}

\section{Centralizer of the isometric $\Gex$-action}
\label{sec:centralizer}

In this section we specialize some known results for actions of non-compact simple Lie groups to the our case of $\Gex$-actions. Our main references are \cite{OQ-SO} and \cite{OQ-U}.

We will assume the hypotheses of Theorem~\ref{thm:diff_type} through out this section. Under such conditions, it is well known that the $\Gex$-action on $M$ is everywhere locally free (see \cite{Szaro}). Hence, the set of orbits defines a foliation $\OO$ on $M$, whose tangent bundle will be denoted by $T\OO$. In particular, the map $M \times \gex \rightarrow T\OO$ given by the assignment $(x, X) \mapsto X^*_x$ is an isomorphism of bundles. We recall that for $X \in \gex$ we denote by $X^*$ the vector field on $M$ whose local flow is $\exp(tX)$. Also, we will denote by $T\OO^\perp$ the bundle whose fibers are the subspaces orthogonal to the fibers of $T\OO$. Then, the condition $\dim(M) \leq 21$ ensures that both $T\OO$ and $T\OO^\perp$ are non-degenerate and so that $TM = T\OO\oplus T\OO^\perp$ (see Lemma~1.4 from \cite{OQ-SO}). In what follows, we will use the same symbols $\OO$, $T\OO$ and $T\OO^\perp$ for the corresponding objects on $\widetilde{M}$.

The following result is fundamental for our work. See Proposition~2.3 from \cite{QuirogaCh} for a proof for arbitrary non-compact simple Lie groups (see also \cite{Gromov,ZimmerAut}). For a pseudo-Riemannian manifold $N$ we denote by $\Kill(N)$ the Lie algebra of globally defined Killing vector fields on $N$. Also, we will denote by $\Kill_0(N,x)$ the Lie subalgebra of $\Kill(N)$ consisting of those vector fields that vanish at $x$. 

\begin{proposition}
    \label{prop:g(x)}
    For $M$ as above, there is a dense conull subset $A \subset \widetilde{M}$ such that for every $x \in A$ the following properties are satisfied.
    \begin{enumerate}
    \item There is a homomorphism $\rho_x : \gex \rightarrow \Kill(\widetilde{M})$ which is an isomorphism onto its image $\rho_x(\gex)$.
    \item Every element of $\rho_x(\gex)$ vanishes at $x$: $\rho_x(\gex) \subset \Kill_0(\widetilde{M},x)$.
    \item For every $X,Y \in \gex$ we have:
        \[
            [\rho_x(X),Y^*] = [X,Y]^* = -[X^*,Y^*].
        \]
        In particular, the elements in $\rho_x(\gex)$ and their corresponding local flows preserve both $\OO$ and $T\OO^\perp$.
	\end{enumerate}
\end{proposition}

The following local homogeneity result is well known and it is a particular case of Gromov's open dense orbit theorem. For its proof for general actions of non-compact simple Lie groups we refer to \cite{Gromov} and \cite{ZimmerAut}.
\begin{proposition}
    \label{prop:GromovOpenDense}
    For $M$ satisfying the above conditions, there is an open dense conull subset $U \subset \widetilde{M}$ such that for every $x \in U$ the evaluation map $ev_x : \HH \rightarrow T_x \widetilde{M}$ given by $Z \mapsto Z_x$ is surjective.
\end{proposition}

For $A$ as in Proposition~\ref{prop:g(x)}, let $x\in A$ be given and consider the map
\begin{align*}
    \widehat{\rho}_x : \gex &\rightarrow \Kill(\widetilde{M}) \\
    \widehat{\rho}_x(X) &= \rho_x(X) + X^*.
\end{align*}
Then, Proposition~\ref{prop:g(x)}(3) implies that $\widehat{\rho}_x$ is an injective homomorphism of Lie algebras. We will denote its image by $\GG(x)$, which is thus a Lie subalgebra of $\HH$ isomorphic to $\gex$. In particular, the Lie brackets induce a $\gex$-module structure on $\HH$. Furthermore, through the isomorphism $\widehat{\rho}_x$ between $\GG(x)$ and $\gex$ every $\GG(x)$-module can be considered as a $\gex$-module.

Proposition~\ref{prop:GromovOpenDense} allows us to define a $\GG(x)$-module structure on $T_x\widetilde{M}$ through the following construction. 

Let $A$ and $U$ be as in Propositions~\ref{prop:g(x)} and \ref{prop:GromovOpenDense}, respectively. Fix some point $x \in A \cap U$. We consider the map $\lambda_x : \GG(x) \rightarrow \so(T_x \widetilde{M})$ given by 
\[
	\lambda_x(Z)(v) = [Z,V]_x,
\] 
where $V \in \HH$ is such that $V_x = v$. It is easy to see that this is a well defined homomorphism of Lie algebras. Furthermore, it is also known that the evaluation map $ev_x : \HH \rightarrow T_x \widetilde{M}$ is a homomorphism of $\GG(x)$-modules that satisfies $ev_x(\GG(x)) = T_x \OO$. In particular, $T_x\OO$ is a $\GG(x)$-module isomorphic to the $\gex$-module $\gex$. As a consequence the subspace $T_x\OO^\perp$ is a $\GG(x)$-submodule of $T_x\widetilde{M}$.

For $x \in A\cap U$, in the rest of this work we consider $\HH$ and $T_x\widetilde{M}$ endowed with the $\GG(x)$-module structures defined above. 

On the other hand, we denote by $\HH_0(x) = \ker(ev_x)$ which, by the previous remarks, is a $\GG(x)$-submodule of $\HH$. Also, it is clear that $\HH_0(x)$ is a Lie subalgebra of $\HH$ as well. In particular, $\GG(x) + \HH_0(x)$ is a Lie subalgebra of $\HH$ that contains $\HH_0(x)$ as an ideal. Hence, $\HH$ can be considered as a module over $\GG(x) + \HH_0(x)$ through the Lie brackets.

By Proposition~3.5 from \cite{OQ-U} we can extend $\lambda_x$ from $\GG(x)$ to the map
\begin{align*}
    \lambda_x : \GG(x) + \HH_0(x) &\rightarrow \so(T_x\widetilde{M}) \\
    \lambda_x(Z)(v) &= [Z,V]_x,
\end{align*}
where for a given $v \in T_x\widetilde{M}$ we choose $V \in \HH$ such that $V_x = v$. As before, it is proved that $\lambda_x$ is a well defined homomorphism of Lie algebras, thus defining a $\GG(x) + \HH_0(x)$-module structure on $T_x\widetilde{M}$ for which both $T_x\OO$ and $T_x\OO^\perp$ are submodules. Furthermore, the evaluation map $ev_x : \HH \rightarrow T_x\widetilde{M}$ is a homomorphism of $\GG(x) + \HH_0(x)$-modules. In particular, we have a representation 
\begin{align*}
    \lambda_x^\perp : \GG(x) + \HH_0(x) &\rightarrow \so(T_x\OO^\perp)\\
    \lambda_x^\perp(Z) &= \lambda_x(Z)|_{T_x\OO^\perp}.
\end{align*}
Furthermore, by Proposition~3.6 from \cite{OQ-U}, the restriction $\lambda_x^\perp : \HH_0(x) \rightarrow \so(T_x\OO^\perp)$ is injective and its image is both a Lie subalgebra and a $\GG(x)$-submodule of $\so(T_x\OO^\perp)$. 

The fact that $\GG(x) \simeq \gex$ as a Lie algebra allows us to obtain the following decomposition of the centralizer $\HH$.

\begin{proposition}
    \label{prop:HHwithVV(x)}
    Let $A$ and $U$ be as in Propositions~\ref{prop:g(x)} and \ref{prop:GromovOpenDense}, respectively. For a fixed point $x \in A \cap U$ there exists a $\GG(x)$-submodule $\VV(x)$ of $\HH$ such that
    \[
        \HH = \GG(x) \oplus \HH_0(x) \oplus \VV(x), \quad T_x\OO^\perp = ev_x(\VV(x)).
    \]
\end{proposition}

Next we consider the analytic map
\[
    \omega : T\widetilde{M} \rightarrow \gex
\]
given by the orthogonal projection onto $T\OO$ followed by the fiberwise isomorphism $T\OO \rightarrow \gex$ described at the beginning of this section. Let us also consider the analytic $\gex$-valued $2$-form $\Omega$ defined by
\[
    \Omega_x = d\omega_x|_{\wedge^2 T_x\OO^\perp},
\]
for every $x \in \widetilde{M}$. If $X, Y$ are smooth sections of $T\OO^\perp$, then $\omega(X) = \omega(Y) = 0$ and so we have
\[
    \Omega(X,Y) = X(\omega(Y)) - Y(\omega(X)) - \omega([X,Y]) = -\omega([X,Y]),
\]
which implies the following result (see \cite{Gromov,QuirogaCh}).
\begin{lemma}
    \label{lem:TOperp-Omega}
    For $G$ and $M$ as above, assume that $T\widetilde{M} = T\OO \oplus T\OO^\perp$. Then, $T\OO^\perp$ is integrable if and only if $\Omega \equiv 0$.
\end{lemma}

By Lemma~\ref{lem:so(E)}, we obtain from the map $\Omega_x : \wedge^2 T_x\OO^\perp \rightarrow \gex$ a corresponding map $\so(T_x\OO^\perp) \rightarrow \gex$ given by $\Omega_x \circ \varphi_x^{-1}$, where $\varphi_x : \wedge^2 T_x\OO^\perp \rightarrow \so(T_x\OO^\perp)$ is the isomorphism defined by Lemma~\ref{lem:so(E)}. This does not change the $\so(T_x\OO^\perp)$-module structure on the domain. Hence, we will denote with the same symbol $\Omega_x$ the linear map given by the $2$-form $\Omega$ when considered as a map $\so(T_x\OO^\perp) \rightarrow \g$.

It turns out that the forms $\omega_x$ and $\Omega_x$ satisfy special intertwining properties with respect to the module structure over $\GG(x) + \HH_0(x)$. These are stated and proved in Proposition~3.10 from \cite{OQ-U} for general non-compact simple Lie group actions. For our given setup, the following hold for every $x \in A \cap U$.
\begin{enumerate}[label=(\thesection.\arabic*)]
	\item\label{item_Omegax} The linear map $\Omega_x : \wedge^2 T_x \OO^\perp \rightarrow \gex$ intertwines the homomorphism of Lie algebras $\widehat{\rho}_x : \gex \rightarrow \GG(x)$ for the actions of $\gex$ on $\gex$ and of $\GG(x)$ on $T_x\OO^\perp$ via $\lambda_x^\perp$. 
    \item\label{item_KerOmegax} The linear map $\Omega_x : \so(T_x\OO^\perp) \rightarrow \gex$ is $\HH_0(x)$-invariant via $\lambda_x^\perp$. In particular, we have
    \[
	    [\lambda_x^\perp(\HH_0(x)), \so(T_x\OO^\perp)] \subset \ker(\Omega_x).
    \]
\end{enumerate}

Given the previous discussion there are two natural cases to consider: either $\Omega \equiv 0$, and $T\OO^\perp$ is integrable, or for some $x \in A\cap U$ the linear map $\Omega_x$ is non-zero, and the above properties for $\Omega_x$ impose strong restrictions on the centralizer $\HH$. The following result provides the description of $\HH$ in the latter case.

\begin{proposition}
    \label{prop:HH_structure}
    For a $\Gex$-action on $M$ as above, for $A$ and $U$ as in Propositions~\ref{prop:g(x)} and \ref{prop:GromovOpenDense}, respectively, let $x \in A \cap U$ be such that $\Omega_x \not= 0$. If $\VV(x)$ is a $\GG(x)$-submodule of $\HH$ given as in Proposition~\ref{prop:HHwithVV(x)}, then $\VV(x) \simeq \R^{3,4}$ as $\gex$-modules and $\dim(M) \leq 21$. Furthermore, $\HH_0(x) = 0$ and $\HH \simeq \so(3,4)$ both as Lie algebras and as modules over $\GG(x) \simeq \gex$.
\end{proposition}
\begin{proof}
	For our given $x$, property~\ref{item_Omegax} implies that $T_x\OO^\perp$ is a non-trivial $\gex$-module. Hence, Proposition~\ref{prop:g22_and_R34} shows that $\VV(x) \simeq T_x\OO^\perp \simeq \R^{3,4}$ as $\gex$-modules. Furthermore, by Proposition~\ref{prop:R34_scalar_product} this isomorphism is an isometry up to a constant. In particular, the representation $\lambda_x^\perp : \GG(x) \rightarrow \so(T_x\OO^\perp)$ discussed above is naturally equivalent to the linear realization $\gex \rightarrow \so(3,4)$.
	
	Hence, Proposition~\ref{prop:wedge2R34_so(R34)_over_g22} yields the existence of a $\GG(x)$-module $V$ of $\so(T_x\OO^\perp)$ such that
    \begin{align*}
        \so(T_x\OO^\perp) &= \lambda_x^\perp(\GG(x)) \oplus V, \\
            [V,V] &= \so(T_x\OO^\perp),
    \end{align*}
    for the brackets of $\so(T_x \OO^\perp)$. Thus, the map $\Omega_x : \so(T_x \OO^\perp) \rightarrow \gex$ is naturally identified with the projection onto the summand $\lambda_x^\perp(\GG(x))$, which implies that $\ker(\Omega_x) = V$. If we apply property \ref{item_KerOmegax} and the fact that $[V,V] = \so(T_x\OO^\perp)$, we conclude that $\lambda_x^\perp(\HH_0(x)) = 0$. By the previous remarks in this section this yields $\HH_0(x) = 0$. 
    
    Hence, we have $\HH = \GG(x) \oplus \VV(x)$ where $\VV(x)$ is a $\GG(x)$-submodule isomorphic through $ev_x$ to $T_x \OO^\perp \simeq \R^{3,4}$. 
    
    If we consider a Levi decomposition $\HH = \LL \oplus \rad(\HH)$ such that $\GG(x) \subset \LL$, then such sum is a decomposition into $\GG(x)$-submodules as well. In particular, either $\rad(\HH) = 0$ or $\rad(\HH) = \VV(x)$. In the latter case, we obtain a semi-direct product $\HH = \GG(x) \ltimes \VV(x)$.

    Suppose that $\rad(\HH) = \VV(x)$ and choose $R$ a simply connected Lie group whose Lie algebra is $\VV(x)$. Hence, the Lie group $\Gex \ltimes R$, with the semi-direct product structure, has Lie algebra $\HH$. Let $\psi : \gex \ltimes \VV(x) \rightarrow \HH$ be the isomorphism whose restriction to $\gex$ is $\widehat{\rho}_x$ and that maps $\VV(x)$ to itself by the identity. By Lemma~1.11 from \cite{OQ-SO} (or by the results from \cite{ONeill}), the completeness of $\widetilde{M}$ and the fact that $\HH \subset \Kill(\widetilde{M})$ imply the existence of a right $\Gex \ltimes R$-action on $\widetilde{M}$ such that
    \[
        \psi(X) = X^*,
    \]
    for every $X \in \gex \ltimes \VV(x)$, where $X^*$ denotes the Killing field generated by the (right) action of the $1$-parameter subgroup $(\exp(tX))_t$ of $\Gex \ltimes R$. Consider the analytic map
    \begin{align*}
        f : \Gex \ltimes R &\rightarrow \widetilde{M} \\
            f(g,r) &= x(g,r),
    \end{align*}
    given by the $\Gex \ltimes R$-orbit at $x$ and which is clearly $\Gex \ltimes R$-equivariant. A straightforward computation (compare with the proof of Proposition~4.4 from \cite{OQ-U}) shows that $df_{(e,e)}$ is an isomorphism that maps
    \[
        df_{(e,e)}(\gex) = T_x \OO, \quad df_{(e,e)}(\VV(x)) = T_x \OO^\perp.
    \]
    In particular, $f$ is a local diffeomorphism from a neighborhood of the identity onto a neighborhood of $x$. If we choose $N = f(\{e\} \times R)$, then $N$ is a submanifold of $\widetilde{M}$ in a neighborhood of $x$ such that
    \[
        T_x N = T_x \OO^\perp.
    \]
    Furthermore, the equivariance of $f$ is easily seen to imply that
    \[
        T_{f(e,r)} N = T_{f(e,r)} \OO^\perp,
    \]
    for every $r$ in a neighborhood of $e$ in $R$. In other words, $N$ is an integral submanifold of $T\OO^\perp$ passing through $x$. Next, the equivariance with respect to $\Gex$ implies that there is an integral submanifold of $T\OO^\perp$ passing through every point in a neighborhood of $x$. By the analyticity of $T\OO^\perp$, we conclude that this vector bundle is integrable. This yields a contradiction since we assumed that $\Omega \not= 0$.

    From the previous discussion we conclude that $\rad(\HH) = 0$ and so that $\HH$ is semisimple. We observe that every decomposition of $\HH$ into simple ideals is also a decomposition into $\GG(x)$-submodules. Since $\HH$ is the sum of the two inequivalent irreducible $\GG(x)$-submodules $\GG(x)$ and $\VV(x)$, if $\HH$ is not simple, then both submodules are ideals. But this is impossible because $[\GG(x), \VV(x)] \not= 0$. We conclude that $\HH$ is a simple Lie algebra of dimension $21$. In particular, $\HH$ is a noncompact real form of the $21$-dimensional simple complex Lie algebra $\HH^\C$. An inspection of the list of simple complex Lie algebras (see \cite{Helgason}) shows that the only possibilities are either $\HH^\C \simeq \so(7,\C)$ or $\HH^\C \simeq \spi(6,\C)$. The latter and the fact that $\gex \simeq \GG(x) \subset \HH$ would imply the existence of a non-trivial $6$-dimensional representation of $\g_2^\C$, which is absurd. We conclude that $\HH^\C \simeq \so(7,\C)$, and so that $\HH \simeq \so(p,q)$ for some $p,q \geq 1$ such that $p + q = 7$. Considering again the inclusion $\gex \simeq \GG(x) \subset \HH$ we obtain a non-trivial representation $\gex \rightarrow \so(p,q)$. Then, Propositions~\ref{prop:g22_and_R34} and \ref{prop:R34_scalar_product} imply that we must have $\{p,q\} = \{3,4\}$ and so we can in fact assume that $p = 3, q = 4$. In other words, we conclude that $\HH \simeq \so(3,4)$. The arguments also show that the inclusion $\GG(x) \subset \HH$ must correspond to the linear realization $\gex \rightarrow \so(3,4)$ and so the isomorphism $\HH \simeq \so(3,4)$ holds in the sense of $\GG(x)$-modules as well.
\end{proof}

\section{Proof of the main results.}
\label{sec:main_proofs}
In what follows we will assume that the hypotheses of Theorem~\ref{thm:diff_type} hold. We will consider two cases according to whether $T\OO^\perp$ is integrable or not. In the first case, the results from \cite{QuirogaCh} imply that the first conclusion from both Theorems~\ref{thm:diff_type} and \ref{thm:metric_type} hold. So we can assume that the conclusions from Proposition~\ref{prop:HH_structure} hold at some point $x_0$.

Hence, the isomorphism $\widehat{\rho}_{x_0} : \gex \rightarrow \GG(x_0)$ can be extended to an isomorphism
\[
    \psi : \so(3,4) \rightarrow \HH.
\]
As in the proof of Proposition~\ref{prop:HH_structure} and by the geodesic completeness of $\widetilde{M}$, we can apply Lemma~1.11 from \cite{OQ-SO} or the results from \cite{ONeill} to obtain an isometric right $\wtSO_0(3,4)$-action on $\widetilde{M}$ such that
\[
    \psi(X) = X^*,
\]
for every $X \in \so(3,4)$. Recall that $X^*$ is the Killing field obtained from the (right) action of the $1$-parameter subgroup $(\exp(tX))_t$ of $\wtSO_0(3,4)$.

Let us denote by
\begin{align*}
    \varphi : \wtSO_0(3,4) &\rightarrow \widetilde{M} \\
        g &\mapsto x_0 g,
\end{align*}
the $\wtSO_0(3,4)$-orbit map at $x_0$. From the previous remarks it follows that
\[
    d\varphi_e(X) = ev_{x_0}(\psi(X))
\]
for every $X \in \wtSO_0(3,4)$, and so defines an isomorphism. Since $\varphi$ is $\wtSO_0(3,4)$-equivariant, we conclude that $\varphi$ is a local diffeomorphism.

Let us denote with $K$ the Killing form of $\so(3,4)$ and let $h_K$ be the bi-invariant pseudo-Riemannian metric on $\wtSO_0(3,4)$ induced by $K$. It is well known that $\wtSO_0(3,4)$ is complete with the pseudo-Riemannian metric $h_K$.

Let $V$ be the $\gex$-submodule of $\so(3,4)$ complementary to $\gex$, as given by Proposition~\ref{prop:wedge2R34_so(R34)_over_g22}. We have proved that $d\varphi_e = ev_{x_0} \circ \psi$ and so it defines an isomorphism of modules from $\so(3,4)$ onto $T_{x_0} \widetilde{M}$ for the module structures over $\gex$ and $\GG(x_0)$, respectively, and with respect to the isomorphism $\widehat{\rho}_{x_0} : \gex \rightarrow \GG(x_0)$. Furthermore, we also have
\[
	d\varphi_e(\gex) = T_{x_0} \OO, \quad d\varphi_e(V) = T_{x_0} \OO^\perp.
\]
On the other hand, the restrictions of the metric $h_{x_0}$ to both $T_{x_0}\OO$ and $T_{x_0}\OO^\perp$ are non-degenerate and $\GG(x_0)$-invariant. It follows that the bilinear forms
\[
	\varphi_e^*(h|_{T_{x_0}\OO}), \quad \varphi_e^*(h|_{T_{x_0}\OO^\perp})
\]
on $\gex$ and $V$, respectively, are non-degenerate and $\gex$-invariant. By Proposition~\ref{prop:R34_scalar_product} there exists non-zero constants $c_1, c_2$ such that
\[
	K|_{\gex} = c_1 \varphi_e^*(h|_{T_{x_0}\OO}), \quad K|_V = c_2 \varphi_e^*(h|_{T_{x_0}\OO^\perp}).
\]
If we consider the pseudo-Riemannian metric on $M$ given by
\[
	\overline{h} = c_1 h|_{T\OO} \oplus c_2 h|_{T\OO^\perp},
\]
then the above discussion shows that the map
\[
	d\varphi_e : (\so(3,4), K) \rightarrow (T_{x_0} \widetilde{M}, \overline{h}_{x_0})
\]
is an isometry. Furthermore, the equivariance of $\varphi$ implies that it defines a local isometry $(\wtSO(3,4), h_K) \rightarrow (\widetilde{M}, \overline{h})$. The completeness of $h_K$ and the results from \cite{ONeill} prove that $\varphi$ is in fact an isometry.

Let us consider the (left) $\Gex$-action on $\widetilde{M}$ lifted from the corresponding action on $M$. This yields from the isometry $\varphi$ a homomorphism
\[
    \rho : \Gex \rightarrow \Iso_0(\wtSO_0(3,4), h_K).
\]
The latter group of isometries is given by $L(\wtSO_0(3,4)) R(\wtSO_0(3,4))$, the group of left and right translations of $\wtSO_0(3,4)$, and so we obtain a pair of homomorphisms
\[
    \rho_1, \rho_2 : \Gex \rightarrow \wtSO(3,4),
\]
such that
\[
    \rho(g) = L_{\rho_1(g)} R_{\rho_2(g)^{-1}},
\]
for every $g \in \Gex$. We note that this $\Gex$-action commutes with the right $\wtSO_0(3,4)$-action on $\widetilde{M}$ and so both actions commute when acting on $\wtSO_0(3,4)$. This implies that
\[
    \rho_2(\Gex) \subset Z(\wtSO_0(3,4)),
\]
thus showing that $\rho_2 = e$. In particular, we have $\rho = L_{\rho_1}$, i.e.~the $\Gex$-action defined by $\rho$ is given by left translations by $\rho_1$. Hence, $\varphi$ is $\Gex$-equivariant for the $\Gex$-action on the domain given by the non-trivial homomorphism $\rho_1 : \Gex \rightarrow \wtSO_0(3,4)$ and left translations. 

Let us identify $\widetilde{M}$ with $\wtSO_0(3,4)$ through the isometry $\varphi$. By the previous discussion we have
\[
    \pi_1(M) \subset \Iso(\wtSO_0(3,4), h_K).
\]
Since $\Iso_0(\wtSO_0(3,4))$ has finite index in $\Iso(\wtSO_0(3,4))$ (see for example \cite{QuirogaAnnals}) we conclude that the discrete subgroup
\[
    \Gamma_1 = \pi_1(M) \cap \Iso_0(\wtSO_0(3,4), h_K) = \pi_1(M) \cap L(\wtSO_0(3,4)) R(\wtSO_0(3,4))
\]
is a finite index subgroup of $\pi_1(M)$. Every element $\gamma \in \Gamma_1$ corresponds to an isometry
\[
    \gamma = L_{g_1} R_{g_2},
\]
where $g_1, g_2 \in \wtSO_0(3,4)$. Since the $\Gamma_1$-action and the lifted $\Gex$-action on $\widetilde{M}$ commute with each other, it follows that $g_1 \in Z = Z_{\wtSO_0(3,4)}(\rho_1(\Gex))$, the centralizer in $\wtSO_0(3,4)$ of the image of $\rho_1 : \Gex \rightarrow \wtSO_0(3,4)$. Hence, we conclude that
\[
    \Gamma_1 \subset L(Z) R(\wtSO_0(3,4)).
\]
We now prove the following.

\begin{lemma}
    \label{lem:Z_G22_SO(3,4)}
    For every non-trivial homomorphism $\rho_1 : \Gex \rightarrow \wtSO_0(3,4)$, the centralizer $Z = Z_{\wtSO_0(3,4)}(\rho_1(\Gex))$ of the image of $\rho_1$ in $\wtSO_0(3,4)$ is a finite subgroup.
\end{lemma}
\begin{proof}
    Consider the corresponding non-trivial homomorphism $d\rho_1 : \gex \rightarrow \so(3,4)$, and let $V$ a $\gex$-submodule of $\so(3,4)$ complementary to $\gex$. By Proposition~\ref{prop:g22_and_R34} we have $V \simeq \R^{3,4}$ as $\gex$-modules. It follows that $d\rho_1(\gex)$ is a maximal subalgebra of $\so(3,4)$. Since $d\rho_1(\gex) + \z$ is a Lie subalgebra of $\so(3,4)$, where $\z$ is the Lie algebra of $Z$, we conclude that $\z = 0$. Hence, $Z$ is a discrete subgroup. Finally, Lemma~1.1.3.7 from \cite{Warner} implies that $Z$ is contained in any maximal compact subgroup of $\wtSO_0(3,4)$. Hence, $Z$ is a finite subgroup.
\end{proof}

By Lemma~\ref{lem:Z_G22_SO(3,4)} we conclude that
\[
    \Gamma = \Gamma_1 \cap R(\wtSO_0(3,4)),
\]
is a finite index subgroup of $\Gamma_1$ and so of $\pi_1(M)$. The group $\Gamma$ is clearly identified with a discrete subgroup of $\wtSO_0(3,4)$ such that 
\[
    \pi : \widehat{M} = \wtSO_0(3,4)/\Gamma \rightarrow \widetilde{M}/\pi_1(M) = M
\]
defines a finite cover of $M$.

On the other hand, a proof similar to that of Lemma~3.4 from \cite{OQ-SO} shows that $M$ has finite volume on the metric $\overline{h}$. Hence, $\Gamma$ is a lattice of $\wtSO_0(3,4)$. This proves that the cases (2) of Theorems~\ref{thm:diff_type} and \ref{thm:metric_type} are satisfied, thus completing the proof of these theorems.

\end{document}